\documentclass{amsart}

\usepackage{amssymb,hyperref,enumerate}

\usepackage[msc-links]{amsrefs}
\usepackage{graphicx}

\newtheorem{thm}{Theorem}[section]
\newtheorem{cor}[thm]{Corollary}
\newtheorem{lem}[thm]{Lemma}
\newtheorem{prop}[thm]{Proposition}

\theoremstyle{definition}
\newtheorem{defn}[thm]{Definition}
\newtheorem{example}[thm]{Example}
\theoremstyle{remark}
\newtheorem{rem}[thm]{Remark}

\newcommand{\ra}{\rightarrow}
\newcommand{\iG}{\overleftarrow{G/\Gamma_n}}
\newcommand{\iH}{\overleftarrow{H/\Gamma_n}}
\newcommand{\iA}{\overleftarrow{\mathbb Z^2/A_n\mathbb Z^2}}

\newcommand{\N}{{\mathbb N}}

\newcommand{\Z}{{\mathbb Z}}

\newcommand{\ha}{{\hat a}}
\newcommand{\hb}{{\hat b}}
\newcommand{\hc}{{\hat c}}
\newcommand{\hd}{{\hat d}}

\newcommand{\mF}{{\mathfrak F}}

\renewcommand{\Im}{\text{Im}}

\begin{document}

\title{The structure and spectrum of Heisenberg odometers}

\author{Samuel Lightwood}
\address{Department of Mathematics, Western Connecticut State University, Danbury, CT 06810}
\email{lightwoods@wcsu.edu}
\author{Ay\c se \c Sah\. in}
\address{Department of Mathematical Sciences, DePaul University, Chicago, IL 60614}
\email{asahin@depaul.edu}
\author{Ilie Ugarcovici}
\address{Department of Mathematical Sciences, DePaul University, Chicago, IL 60614}
\email{iugarcov@depaul.edu}


\date{July 30, 2012}

\begin{abstract}
Odometer actions of discrete, finitely generated and residually finite groups $G$ have been defined by Cortez and Petite.  In this paper we focus on the case where $G$ is the discrete Heisenberg group.  We prove a structure theorem for finite index subgroups of the Heisenberg group based on their geometry when they are considered as subsets of $\mathbb Z^3$.  We use this structure theorem to provide a classification of Heisenberg odometers and we construct examples of each class.  In order to construct some of the examples we also provide necessary and sufficient conditions for a $\mathbb Z^d$ odometer to be a product odometer as defined by Cortez.
It follows from work of Mackey that all such actions have discrete spectrum.
Here we provide a different proof of this fact for general $G$ odometers which allows us to identify explicitly those representations of the Heisenberg group which appear in the spectral decomposition of a given Heisenberg odometer.
 \end{abstract}

\maketitle

\section{Introduction}
In the classical theory, discrete-time measurable and topological dynamical systems are generated by the iteration of a single automorphism on a space, and are thus actions of the group $\mathbb Z$.  Since the 1970's an important direction of research has focused on actions of more general groups.  The work of Ornstein and Weiss \cite{OW} has established countable amenable groups as a general setting for ergodic theory.  The discrete Heisenberg group is the first natural example of a non-abelian group in this category and its actions are of interest in a variety of contexts within and without dynamical systems.  While the work in \cite{OW} establishes many tools of ergodic theory for the discrete Heisenberg group, actions of this group are not nearly as well understood as actions of $\mathbb Z^d$.  This is due, partly, to the non-commutative nature of the group but also to the more complicated geometry of orbits of the action.    Recent work has focused on algebraic actions of this group (e.g. \cite{ER},\cite{DK}).  

In this paper, we provide some geometric intuition for general measurable actions of this group by studying the concrete case of Heisenberg odometers.  Our main result gives a detailed geometric analysis of these actions.  We also provide an explicit description of the finite dimensional,  irreducible representations of the Heisenberg group which can arise in the spectral decomposition of a Heisenberg odometer.  Our work establishes explicitly, generalizing the classical theory, the connection between the algebraic structure of the odometer and the spectrum of the action.

\subsection*{An overview of odometer actions}
Odometer  systems are a well studied class of examples in the classical theory of measurable and topological dynamical systems generated by a single transformation.  They are rank one transformations, and therefore are ergodic and have zero entropy.  They have discrete rational spectrum and are the key ingredients in the study of Toeplitz systems.
They can be viewed  measure theoretically as cutting and stacking transformations of the unit interval.  Alternatively, they can be viewed algebraically as an action of $\mathbb Z$ on an inverse limit space of increasing quotient groups of $\mathbb Z$.  
They can also be viewed as an action by addition in an adic group.  It follows that they are, in fact uniquely ergodic  (see, for example, \cite{D05} and the references therein, and \cite{N98}).

All of these perspectives can be generalized to define odometer actions of $\mathbb Z^d$ and there is an obvious way to construct examples.  In the case of $d=2$, given any two $\mathbb Z$ odometer actions $(T,X)$ and $(S,Y)$,  the maps $T\times Id, Id\times S$ acting on $X\times Y$ clearly commute and satisfy the appropriate generalizations of the above ideas to $\mathbb Z^2$.
In \cite{Co06} Cortez defines odometer actions of $\mathbb Z^d$ using the inverse limit approach.  The author calls the obvious examples described above {\it product odometers} and gives an example of a non-product type $\mathbb Z^2$ odometer.  As in the classical case, $\mathbb Z^d$ odometers are uniquely ergodic and have zero entropy. In \cite{CP08} Cortez and Petite generalize the work in \cite{Co06} to define $G$ odometer dynamical systems for $G$ any discrete, finitely generated and residually finite group, and show that they are also uniquely ergodic.   In Section~\ref{sec:God} we provide this definition and introduce the notation to be used in the paper.

\subsection*{Heisenberg odometers}
Let $H$ be the discrete Heisenberg group,  defined on the set $\mathbb Z^3$ with the following  group multiplication:
\begin{equation*}
(x,y,z)(x',y',z')=(x+x',y+y',z+z'+xy').
\end{equation*}

In this paper we show that there are geometric considerations similar to the $\mathbb Z^2$ case which separate different types of Heisenberg odometers. 
 Informally, thinking of subgroups of the Heisenberg group as subsets of $\mathbb Z^3$ with a different group multiplication, 
we can associate to any Heisenberg odometer, a $\mathbb Z^2$ odometer constructed by considering the projections of the Heisenberg subgroups onto their first two coordinates.  If the associated $\mathbb Z^2$ odometer is of product type, as defined by Cortez, we call the Heisenberg odometer an
{\em $(x,y)$-product}  odometer.  
If the subgroups used to construct the odometer, considered as sets in $\mathbb Z^3$, have the structure $A\mathbb Z^2\times m\mathbb Z$ for some nonsingular matrix $A\in M(2,\mathbb Z)$ and $m\in\mathbb N$  we call it a {\em flat} odometer. 
If the Heisenberg odometer is both of $(x,y)$-product type and flat we call it a {\em pure product} odometer.    

As in the $\mathbb Z^d$ case, it is obvious how to construct pure product type odometers for the Heisenberg group.  We show by construction that  all the other types of odometers also exist.

In Section~\ref{sec:subgrp} we provide an analysis of the structure of the subgroups of the Heisenberg group necessary to classify odometer actions and in Section~\ref{sec:heisclassify} we complete the work outlined above.  Figure~\ref{fig:pic} provides a guide to the examples constructed in the paper.  In order to construct Heisenberg odometers that are not of product type in Section~\ref{sec:heisclassify} we  extend the work in \cite{Co06} by giving necessary and sufficient conditions for a $\mathbb Z^d$ odometer to be of product type.  This characterization allows us to identify a large collection of non-product examples. 
   
\begin{figure}[t]
\includegraphics[scale=0.7]{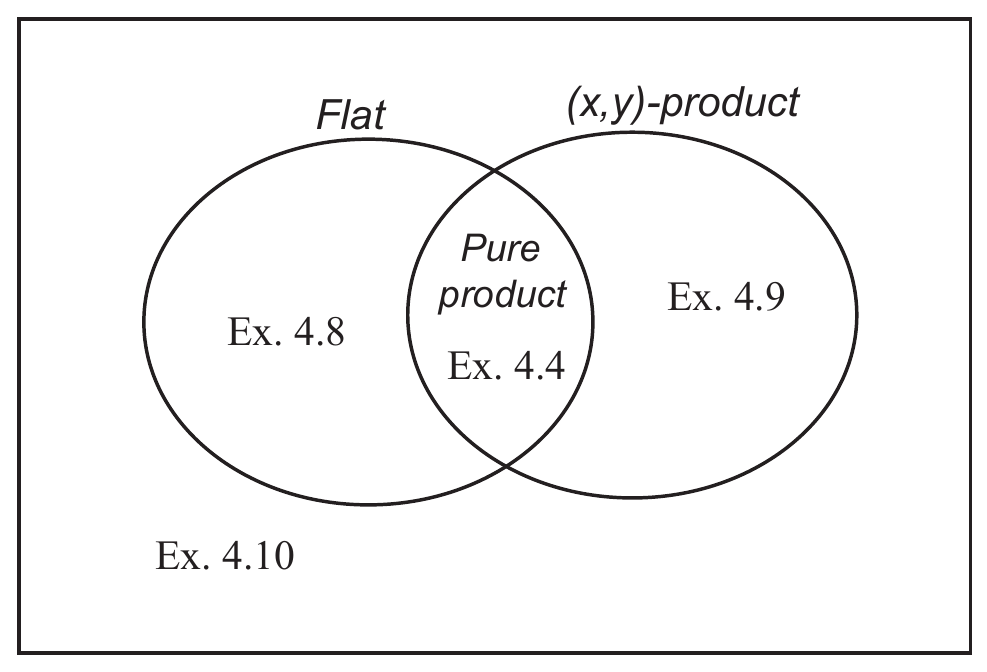}
\caption{Examples of all possible classes of Heisenberg odometers}
\label{fig:pic}
\end{figure}

\subsection*{Spectral analysis of $G$ odometers}
Let $G$ be as described above and recall that an ergodic and measure preserving action of $G$ is said to have discrete spectrum if the associated unitary representation can be decomposed into a direct sum of irreducible, finite dimensional representations of $G$.   In \cite{M64} Mackey shows that any action of this class of groups $G$ which is conjugate to a rotation on a compact group by a dense subgroup has discrete spectrum as a consequence of the Peter-Weyl Theorem.  A $G$ odometer is an example of this type of action.  Here we present a different argument where we explicitly construct the decomposition into irreducible finite dimensional representations of $G$. As an easy corollary, for those groups $G$ where entropy theory has been sufficiently developed, we have that $G$ odometers have zero entropy.

In the case of Heisenberg odometers, our analysis of the geometric structure of the  subgroups of $H$ allows us to give a complete description in Section~\ref{sec:spectral} of the finite dimensional, irreducible representations of $H$ that can occur in the spectral decomposition of any given Heisenberg odometer action.

\section{Defining $G$ odometers}\label{sec:God}
Let $G$ be a discrete, finitely generated and residually finite group.  Following \cite{CP08} we define a $G$ odometer dynamical system as follows.  Since $G$ is residually finite, there exists a sequence 
$\Gamma_1\supset \Gamma_2\supset\cdots \Gamma_n\supset\cdots$
of subgroups with finite indices in $G$ such that $\cap \Gamma_n = \{e\}$. 
Let  
$\pi_n\colon G/\Gamma_{n+1}\rightarrow G/\Gamma_n$ 
be the map induced by the inclusion $\Gamma_{n+1}\subset \Gamma_n$ and denote by $\iG$ the inverse limit space of the sequence $\{(G/\Gamma_n,\pi_n)\}_{n\geq 1}$. It is a compact metrizable space  whose topology is spanned by the cylinder sets
\[
[n;\gamma]=\{\mathbf{g}\in \iG : g_n=\gamma\} \text{ with } \gamma\in G/\Gamma_n.
\]

The group $G$ acts by left multiplication on $\iG$.  If the subgroups $\{\Gamma_n\}$ are normal in $G$, the system is called a $G$ odometer and otherwise a subodometer.  The odometer action of $G$ on $\iG$ preserves Haar measure $\mu$ and is uniquely ergodic.  We refer to the dynamical system $(\iG,\mu,G)$ as the $G$ odometer on $\iG$.

There exists an effective criterion to check whether two $G$ odometers are conjugate which we will use extensively. It is based on the following characterization of a factor map:
\begin{lem}[\cite{CP08}]\label{lem:factor}  There exists a factor map $\pi:
\overleftarrow{G/\Gamma^1_n} \rightarrow \overleftarrow{G/\Gamma^2_n}$ between two $G$ odometers if and only if for every $\Gamma^2_n$ there exists $\Gamma^1_{k}$ such that  $\Gamma^1_{k}\subset \Gamma^2_{n}$. 
\end{lem}

\section{Subgroups of the discrete Heisenberg group}\label{sec:subgrp}
In this section we describe the geometry of subgroups of $H$, the discrete Heisenberg group, in terms of their projection onto their first two coordinates, and the structure of the fiber over this projection.  The work in this section will allow us to carry out the classification of Heisenberg odometers described in the introduction.  

The following facts are easily verified by computation and we will use them frequently in our arguments:
\begin{equation}\label{conj}
\begin{split}
&(x,y,z)^{-1}=(-x,-y,-z+xy)\\
&(u,v,w)(x,y,z)(u,v,w)^{-1}=(x,y,z+uy-vx)\\
&(u,v,w)(x,y,z)(u,v,w)^{-1}(x,y,z)^{-1}=(0,0,uy-vx).
\end{split}
\end{equation}

Recall also that there exist group homomorphisms $f:\mathbb Z\rightarrow H$, $z\mapsto (0,0,z)$, and $g:H\rightarrow\mathbb Z^2$, $(x,y,z)\mapsto (x,y)$, so that the following is a short exact sequence:
\begin{equation}
0\rightarrow\mathbb Z=<(0,0,1)>\overset{f}{\rightarrow}H\overset{g}{\rightarrow}\mathbb Z^2\rightarrow 0.\label{eq:short}
\end{equation}

\begin{prop}\label{prop:struc}  Let $\Gamma$ be a finite index subgroup of $H$ and let 
$f, g$ be as in \eqref{eq:short}.   Let   $m_\Gamma=[\Im f : \Gamma\cap \Im f]$, and $A\in M(2,\mathbb Z)$ be such that $A\mathbb Z^2= g(\Gamma)$.  Then $A$ is nonsingular and there exists
a map $i_\Gamma:A\mathbb Z^2\rightarrow\mathbb Z$ defined by
\begin{equation*} 
i_\Gamma(x,y)=\min\{z\ge0:(x,y,z)\in\Gamma\}
\end{equation*}
such that 
\begin{equation}\label{struct}
\Gamma=\{(x,y,i_\Gamma(x,y)+km_\Gamma): (x,y)\in A\mathbb Z^2, k\in\mathbb Z\}.
\end{equation}
\end{prop}
 
\begin{proof} 
Since $\Gamma$ is a finite index subgroup of $H$ we have $m_\Gamma<\infty$ and
\begin{equation}\label{eq:mgamma}
\Gamma\cap \Im f=\{0\}\times\{0\}\times m_\Gamma\mathbb Z.
\end{equation} 
Also $g(\Gamma)=A\mathbb Z^2$ is a finite index subgroup of $\mathbb \Z^2$, hence $A$ is a nonsingular matrix.

Since $(0,0,km_\Gamma)$ and $(x,y,i_\Gamma(x,y))$ are both in $\Gamma$, it follows immediately from the group operation that $(x,y,i_\Gamma(x,y)+km_\Gamma)\in \Gamma$. Now let $(x,y,z)\in \Gamma$. Then
$$
(x,y,i_\Gamma(x,y))(x,y,z)^{-1}=(0,0,i_\Gamma(x,y)-z)\in \Gamma\cap \Im f.
$$
This implies that $z\in i_\Gamma(x,y)+m_\Gamma \mathbb Z$, concluding the proof.
\end{proof}

The next proposition identifies some algebraic properties of the function $i_\Gamma$, and the relationship between the constant $m_\Gamma$ and the matrix $A$.

\begin{prop}\label{prop:normal}
Let $\Gamma$ be a finite index subgroup of $H$, and $A, i_\Gamma$, $m_\Gamma$ be as in Proposition~\ref{prop:struc}.  Then
\begin{equation}\label{iprop}
i_\Gamma(x,y)+i_\Gamma(x',y')+xy'=i_\Gamma(x+x',y+y')\mod m_\Gamma
\end{equation}
and
\begin{equation}\label{ndivdet}
m_\Gamma\vert det(A).
\end{equation}
If, in addition, $\Gamma$ is a normal subgroup then we have the stronger conclusion that $m_\Gamma$ divides all the entries in the matrix $A$.
 \end{prop} 

\begin{proof}

Let $(x,y,i_\Gamma(x,y)+km_\Gamma),(x',y',i_\Gamma(x',y')+k'm_\Gamma)\in\Gamma$.  Proposition~\ref{prop:struc} yields that 
\begin{equation*}
i_\Gamma(x,y)+i_\Gamma(x',y')+(k+k')m_\Gamma+xy'=i(x+x',y+y')+k''m_\Gamma
\end{equation*}
for some $k''\in\mathbb Z$, and \eqref{iprop} follows.
To see that (\ref{ndivdet}) holds choose $p,q,r,s\in\mathbb Z$ such that $m_\Gamma$ is relatively prime to $ps-rq$. Let 
$\begin{pmatrix}
u & x\\v & y
\end{pmatrix}=A\begin{pmatrix}
p & r \\q & s
\end{pmatrix}$.
Then both $(u,v),(x,y)\in A\mathbb Z^2$ and therefore there exist $w,z\in\mathbb Z$ such that $(u,v,w)$, $(x,y,z)\in\Gamma$.  Using \eqref{conj} we then have that  $(0,0,uy-vx)\in\Gamma\cap\Im f$ and \eqref{eq:mgamma} implies that $m_\Gamma$ must divide $uy-vx$.  
But
\[
uy-vx=\det \begin{pmatrix}
u & x\\v & y
\end{pmatrix}=\det(A)\det\begin{pmatrix}
p & r \\q & s
\end{pmatrix}
\]
so by our choice of $p,q,r,s$ we must have that $m_\Gamma$ divides $det(A)$.  

Finally, suppose that $\Gamma$ is normal, and choose any $(x,y,z)\in\Gamma$ and $(u,v,w)\in H$.  Normality, the definition of $m_\Gamma$, and  (\ref{conj}) yield that $m_\Gamma | (uy-vx)$ for any $(u,v)\in \Z^2$. This implies that  $m_\Gamma | x,y$. Since $(x,y)\in A\Z^2$ is arbitrary, it follows that $m_\Gamma$ must divide each entry of the matrix $A$.  
\end{proof}

The following proposition gives a converse result to the previous two propositions.

\begin{prop}\label{prop:normalconv}
Consider a nonsingular matrix $A\in M(2,\Z)$, an integer $m\ge1$ that divides the entries of $A$, and a map $i:A\mathbb Z^2\rightarrow \Z$ that satisfies $i(0,0)=0$ and
\begin{eqnarray}
i(x,y)+i(x',y')=i(x+x',y+y')\mod m.\label{ihomom}
\end{eqnarray}
Then
$\Gamma_{A,i,m}=\{(x,y,km+i(x,y)):(x,y)\in A\mathbb Z^2, k\in\mathbb Z\}$
is a finite index normal subgroup of $H$.
\end{prop}
\begin{proof}
If $\Gamma_{A,i,m}$ is a subgroup, our choice of $A$ and $m$ will guarantee that it is a subgroup of finite index.  To see that it is a subgroup
first note that 
since $i(0,0)=0$ we have $(0,0,0)\in\Gamma_{A,i,m}$.  
Choose $\gamma=(x,y,i(x,y)+km)$ for some $k\in\mathbb Z$.  
To see that $\gamma^{-1}$ is also in $\Gamma$ we use (\ref{conj}) to obtain
$
\gamma^{-1}=(-x,-y,-i(x,y)-km+xy).
$
Using (\ref{ihomom}) we can replace $-i(x,y)$ with $i(-x,-y)+k'm$ for some $k'\in\mathbb Z$. Since $m$ divides the entries of $A$, it must divide $x$ and $y$ yielding 
that $\gamma^{-1}$ lies in $\Gamma$.

Closure under addition and normality follow from similar arguments.
\end{proof}
 
\begin{rem}
Note that in Proposition~\ref{prop:struc} if $\Gamma$ is a normal, finite index subgroup then since  $m_\Gamma$ divides the entries of $A$ in \eqref{iprop}  we have $xy'= 0\mod m_\Gamma$.  It is easy to see that the map $i$ defined in \eqref{ihomom} satisfies $i=i_{\Gamma_{A,i,m}}\mod m$ where $i_{\Gamma_{A,i,m}}$ is as defined in Proposition~\ref{prop:struc}.
\end{rem}

The next result will allow us to easily check  if a sequence of normal subgroups of $H$ is nested. 
\begin{prop}\label{prop:odometer}
Suppose $\Gamma$ and $\Gamma'$ are finite index normal subgroups of $H$ defined by the triples $(A_\Gamma\mathbb Z^2,m_{\Gamma},i_{\Gamma})$ and $(A_{\Gamma'}\mathbb Z^2,m_{\Gamma'},i_{\Gamma'})$, respectively. Then $\Gamma'\le \Gamma$ if and only if $A_{\Gamma'}\mathbb Z^2\subset A_\Gamma\mathbb Z^2$, $m_\Gamma | m_{\Gamma'}$ and 
\begin{equation}\label{eq:cond}
i_{\Gamma'}(x,y)=i_\Gamma(x,y) \mod m_\Gamma \text{ for all } (x,y)\in A_{\Gamma'}\Z^2.
\end{equation}
\end{prop}
\begin{proof}
The proof follows immediately when  the subgroups $\Gamma$ and $\Gamma'$ are written in the form given by \eqref{struct} of Proposition~\ref{prop:struc}.
\end{proof}

\section{The Structure of Heisenberg Odometer Actions}\label{sec:heisclassify}
In this section we show that it is possible to describe the structure of Heisenberg odometers in terms of the geometric structure of its defining subgroups.  Using the results of the previous section we can now describe any sequence of subgroups $\Gamma_n$ of $H$ in terms of the  sequence of triples
$(A_n\mathbb Z^2,m_{\Gamma_n},i_{\Gamma_n})$.   
 
 \begin{defn} A finite index subgroup $\Gamma$ of $H$ is called:
\begin{itemize}
\item \emph{flat} if $i_\Gamma\equiv 0$;
\item \emph{$(x,y)$-product} if there exists a diagonal nonsingular matrix $A\in M(2,\Z)$ such that $\Im g_\Gamma=A\Z^2$;
\item \emph{pure product} if it is flat and $(x,y)$-product.
\end{itemize}
\end{defn}

Note that for a given subgroup $\Gamma<H$, the properties of the map $i_\Gamma$ are defined independent of whether $A_\Gamma\mathbb Z^2$ is of product type or not.  Based on these two characteristics, we introduce the following conjugacy invariants for  $H$ odometers.
\begin{defn}
An $H$ odometer on $\overleftarrow{H/\Gamma_n}$ is called a \emph{flat}  
(\emph{$(x,y)$-product, pure product})
 odometer  if  it is conjugate to an $H$ odometer on $\overleftarrow{H/\Gamma'_n}$ where each normal subgroup $\Gamma'_n$ is flat 
 ($(x,y)$-product, pure product).  
 \end{defn}

The remainder of the section is devoted to justifying these classifications. 
Note that if the sequence $\{\Gamma_n\}$ defines a Heisenberg odometer, then the sequence $\{A_n\mathbb Z^2\}$ must also define a $\mathbb Z^2$ odometer on $\overleftarrow{\mathbb Z^2/A_n\mathbb Z^2}$ called the {\em associated $\mathbb Z^2$ odometer}.  

\begin{thm}\label{thm:main}
A Heisenberg odometer must be of one of the following types:  pure product, $(x,y)$-product and not flat, flat and not $(x,y)$-product, and neither flat nor $(x,y)$-product.  Furthermore there exist Heisenberg odometers of each type.
\end{thm}
We first construct examples of odometers of each type.
We defer the proof of the rest of the theorem to the end of the section as it relies on lemmas which arise naturally in the construction of the examples.  

\begin{example}[{\bf Pure product Heisenberg odometer}]\label{ex:purep} One of many examples that can be constructed is given by
$A_n=\begin{pmatrix}2^n & 0\\ 0 & 2^n\end{pmatrix}$, $m_n=2^n$, $i_n=0$.
\end{example}

\begin{rem}\label{rem:xyprod}
A direct consequence of Lemma \ref{lem:factor} and Proposition \ref{prop:odometer} is that a Heisenberg odometer is of $(x,y)$-product type if and only if the associated $\mathbb Z^2$ odometer is of product type.  
\end{rem}
In producing  Heisenberg odometers which are not $(x,y)$-product we cannot simply use
Cortez's \cite{Co06} example of a non-product $\Z^2$ odometer.  That example is constructed with the following sequence of matrices:
\begin{equation}\label{cormat}
\begin{pmatrix}
3^{n+1} & 7\cdot 11^n\\
7\cdot 3^n & 11^{n+1}
\end{pmatrix}
\end{equation}
which have the property that each row has entries that are relatively prime, a sufficient condition for the resulting odometer to be non-product type.  However, by Proposition~\ref{prop:normal} any sequence of normal subgroups $\Gamma_n$ of the Heisenberg group with such a projection must have $m_{\Gamma_n}=1$ for all $n$ and therefore will not have trivial intersection.  Thus this family of non-product type $\mathbb Z^2$  examples  cannot be used to construct  Heisenberg odometers.  

In what follows we give a new necessary and sufficient condition for a $\Z^d$ odometer to be  conjugate to a product odometer which allows us to give examples  where the entries of the matrices have a non-trivial common factor.  Examples of flat and non $(x,y)$-product type Heisenberg odometers will follow immediately.  
Given $A\in M(d,\mathbb Z)$ and $m\in\mathbb Z$ we write $m | row_k(A)$ to mean $m$ divides all the entries in the $k$th row of $A$, and we define $m_k(A)=\min\{m\ge 1:m\vec{e}_k\in A\Z^d\}$.
 \begin{prop}\label{thm:zd} 
A $\Z^d$ odometer on $\overleftarrow{\Z^d/A_n\Z^d}$ is conjugate to a product odometer if and only if  the sequence  $\{A_n\}$, or some subsequence of it, satisfies for all $n\in\N$ and $1\leq k\leq d$ 
\begin{equation}
\label{thm:one}
 m_k(A_n)\,| row_k(A_{n+1})  \,. 
\end{equation}
\end{prop}
\begin{proof}
Assume that the odometer on $\overleftarrow{\Z^d/A_n\Z^d}$ is conjugate to the odometer on $\overleftarrow{ \Z^d/\Delta_n\Z^d}$ where each $\Delta_n\in M(d,\Z)$ is a nonsingular diagonal matrix. The inclusion $\Delta_n\Z^d\supset \Delta_{n+1}\Z^d$ is equivalent to
\[
\Delta_n(k,k) | \Delta_{n+1}(k,k)\quad \forall k=1,\dots,d\,.
\]
Also, by Lemma \ref{lem:factor}, and passing as necessary to a subsequence, one has
\begin{equation}\label{eq:incl}
A_n\Z^d\supset \Delta_n\Z^d\supset A_{n+1}\Z^d\,.
\end{equation}
Since $\Delta_n\vec{e}_k=\Delta_n(k,k)\vec{e}_k\in A_n\Z^d$, it follows that $m_k(A_n)|\Delta_n(k,k)$. From the relation $A_{n+1}\Z^d\subset \Delta_n\Z^d$ we also have that $\Delta_n(k,k)|row_k(A_{n+1})$ and \eqref{thm:one} follows.

Conversely, if \eqref{thm:one} holds, one considers the sequence of diagonal matrices $\{\Delta_n\}$ given by $\Delta_n(k,k)=m_k(A_n)$. Notice that $\Delta_n\Z^d\supset A_{n+1}\Z^d$. Moreover,
$\Delta_n\vec{e}_k=m_k(A_n)\vec{e}_k\in A_n\Z^d$, hence $\Delta_n\Z^d\subset A_n\Z^d$. Therefore \eqref{eq:incl} holds, and  the odometers on $\overline{\Z^d/A_n\Z^d}$ and $\overleftarrow{ \Z^d/\Delta_n\Z^d}$ are conjugate.
\end{proof}

For any nonsingular matrix $A\in M(d,\Z)$ one can factor out the greatest common (positive) factor of each row into a diagonal matrix $\Delta$ and write $A=\Delta\cdot \widehat{A}$ where each row of $\widehat{A}\in M(d,\Z)$ has relatively prime entries. 
For $d=2$, one can check that $m_k(A)$ as defined above satisfies $m_k(A)=\Delta(k,k)\cdot|\det(\widehat{A})|\,.$
Thus we have the following immediate corollary:
\begin{cor}\label{cor:nonprod}
A $\Z^2$ odometer on $\overleftarrow{\Z^2/A_n\Z^2}$ is conjugate to a product odometer if and only if the sequence  $\{A_n\}$, or some subsequence of it, satisfies for all $n\in\N$ and $1\leq k\leq 2$
\begin{equation}
\Delta_n(k,k)\cdot \det(\widehat{A}_n) \,| row_k(A_{n+1}).
\end{equation}
\end{cor}

We can now construct a flat Heisenberg odometer which is not $(x,y)$-product.

\begin{example}[{\bf Flat but not $(x,y)$-product type Heisenberg odometer}]

We modify the matrices in \eqref{cormat} slightly:
\begin{equation}
A_n=\begin{pmatrix}
2^n & 0\\
0 & 2^n
\end{pmatrix}
\cdot\begin{pmatrix}
3^{n+1} & 7\cdot 11^n\\
7\cdot 3^n & 11^{n+1}
\end{pmatrix}
\end{equation}
and we consider the subgroups $\Gamma_n$ given by
the triple $(A_n\mathbb Z^2,2^n,0)$.
By Proposition \ref{prop:normalconv}  each $\Gamma_n$ is a normal subgroup of $H$. It is easy to check that  $\Gamma_{n+1}\subset \Gamma_n$ and $\cap \Gamma_n=\{(0,0,0)\}$, hence the $H$ odometer on $\overleftarrow{H/\Gamma_n}$ is flat. 
By Corollary~\ref{cor:nonprod} the associated $\mathbb Z^2$ odometer is not of product type.  Indeed, $\Delta_n(1,1)\cdot\det(\hat A_n)=2^n\cdot (-16)\cdot 33^n$, and  $A_m(1,1)=2^m\cdot 3^{m+1}$ and there are no $m>n$ so that $2^n\cdot (-16)\cdot 33^n$ divides $2^m\cdot 3^{m+1}$. 
\end{example}

Given any sequence $\{A_n\mathbb Z^2\}$ it is not obvious how to choose a non-trivial sequence $i_n$ so that the resulting sequence of subgroups will give rise to a non-flat odometer, or if it is possible to do so even for product type odometers.  Below we provide an example of an $(x,y)$-product but not flat Heisenberg odometer.

\begin{example}[{\bf $(x,y)$-product but not flat Heisenberg odometer}]\label{ex:kn}
 Consider the sequence of positive integers $\{k_n\}$ defined recursively by $k_1=2$ and  $k_{n+1}=k_n(k_n+1)$. For each $n$, we construct by Proposition \ref{prop:normal} the normal subgroup $\Gamma_n$ using the diagonal matrix $A_n=\left(\begin{array}{cc}k_n & 0\\0 & k_n\end{array}\right)$, positive integer $m_n=k_n$ and the map $i_n:A_n\Z^2\ra\Z$ defined as $i_n(x,y)=x/k_n$. Notice that 
 \begin{equation*}
 \Gamma_n=\{(k_nu,k_nv,k_nw+u):u,v,w\in\Z\}.
 \end{equation*}
We now check $\Gamma_{n+1}\subset \Gamma_n$ using Proposition~\ref{prop:odometer}. The only non-trivial condition to verify is \eqref{eq:cond}: if $x=k_{n+1}u$, then
$i_{{n+1}}(x,y)=u$ and $i_{n}(x,y)=(k_{n}+1)u=u \mod k_n$, as needed.

Also $\cap \Gamma_n=\{(0,0,0)\}$, otherwise if $(0,0,0)\ne (x,y,z)\in \cap \Gamma_n$, then at least one of $|x|,|y|,|z|\ge k_n$ for all $n\ge 1$, which is impossible since $k_n\ra \infty$. Therefore the $H$ odometer on $\overleftarrow{H/\Gamma_n}$ is of $(x,y)$-product type.

Notice that  $(k_n,k_n,1)\in\Gamma_n$ for every $n\in \N$.  If this example has an odometer factor on $\overleftarrow{H/\Gamma'_n}$ where each subgroup $\Gamma'_n$ is flat, then by Lemma \ref{lem:factor}  it would follow that for any $\Gamma'_n$ there exists $\Gamma_{m}$ such that $ \Gamma_{m}\subset \Gamma'_n$, and so $(k_m,k_m,1)\in \Gamma'_n$, as well. 

Now suppose that $i_{\Gamma'_n}\equiv 0$.   
This contradicts the fact that $\Gamma'_n$ is flat: for if $i_{\Gamma'_n}\equiv 0$ and there is an element of the form $(k_m,k_m,1)\in\Gamma'_n$ for all $n$, then Proposition~\ref{prop:struc} implies that $(0,0,1)\in\Gamma'_n$ for all $n,$ and therefore the sequence of subgroups do not have trivial intersection.
\end{example}

\begin{example}[{\bf Not flat and not $(x,y)$-product Heisenberg odometer}]  Using the sequence $\{k_n\}$ defined in the previous example and the matrices in~\eqref{cormat}, we consider the sequence of matrices
\begin{equation}
A_n=\begin{pmatrix}
k_n & 0\\
0 & k_n
\end{pmatrix}
\cdot\begin{pmatrix}
3^{n+1} & 7\cdot 11^n\\
7\cdot 3^n & 11^{n+1}
\end{pmatrix},
\end{equation}
and the normal subgroups given by the triples $(A_n, k_n, i_n)$ where $i_n:A_n\Z^2\ra\Z$ is defined as $i_n(x,y)=x/k_n$. We can describe $\Gamma_n$ as 
\[
\{(k_n(3^{n+1}\cdot u+7\cdot 11^{n}\cdot v),k_n(7\cdot 3^n\cdot u+11^{n+1}\cdot v), k_n\cdot w+3^{n+1}\cdot u+7\cdot 11^{n}\cdot v):u,v,w\in\Z\}.
\]
The inclusion $\Gamma_{n+1}\subset \Gamma_n$ follows from Proposition~\ref{prop:odometer}: in order to verify \eqref{eq:cond} notice that if $x=k_{n+1}(3^{n+2}u+7\cdot 11^{n+1}v)$, then
$i_{{n+1}}(x,y)=3^{n+2}u+7\cdot 11^{n+1}v$ and $i_{n}(x,y)=(k_{n}+1)(3^{n+2}u+7\cdot 11^{n+1}v)=3^{n+2}u+7\cdot 11^{n+1}v\mod k_n$, as needed.

An argument similar to that of  the previous example shows that  $\cap \Gamma_n=\{(0,0,0)\}$, hence the sequence$\{\Gamma_n\}$ defines a Heisenberg odometer. 

By choosing $u,v\in \Z$ such that $3^{n+1}\cdot u+7\cdot 11^{n}\cdot v=1$ and letting $w=0$, we have  $(k_n,k_n,1)\in\Gamma_n$ for every $n\in \N$. We conclude, as above, that the odometer on $\overleftarrow{H/\Gamma_n}$ cannot be flat.

We analyze now the associated $\Z^2$ odometer, $\overleftarrow{\Z^2/A_n\Z^2}$ and show that the condition stated in Corollary \ref{cor:nonprod} is not satisfied. Indeed, $\Delta_n(1,1)\cdot\det(\hat A_n)=k_n\cdot (-16)\cdot 33^n$, and  $A_m(1,1)=k_m\cdot 3^{m+1}$ and there are no $m>n$ so that $k_n\cdot (-16)\cdot 33^n$ divides $k_m\cdot 3^{m+1} $, since each integer $k_m=2\cdot odd$.  
\end{example}

\begin{proof}[Proof of of Theorem~\ref{thm:main}]
We need only verify that if a Heisenberg odometer is both flat and $(x,y)$-product then it is pure product.  Suppose the odometer on
$\overleftarrow{H/\Gamma_n}$ is flat and $(x,y)$-product. Let $\overleftarrow{H/\Gamma'_n}$ be the conjugate flat odometer defined by $(A'_n,m'_n,0)$. By transitivity of conjugacy the odometer on $\overleftarrow{H/\Gamma'_n}$ must also be $(x,y)$-product type and, as noted in  Remark  \ref{rem:xyprod},  this implies that the associated $\mathbb Z^2$ odometer is product type.  Using (\ref{eq:incl}) we have diagonal matrices $\Delta_n'$ such that (after passing to a subsequence, if necessary)
$A'_n\Z^2\supset \Delta'_n\Z^2\supset A'_{n+1}\Z^2\,.$
We claim that the pure product odometer on $\overleftarrow{H/\tilde\Gamma'_n}$ defined by $(\Delta'_n, m'_n,0)$ is conjugate to the odometer on $\overleftarrow{H/\Gamma'_n}$, and therefore to the one  on $\overleftarrow{H/\Gamma_n}$.  This follows immediately from Lemma~\ref{lem:factor} since the property of $\Delta'_n$ guarantees that $\Gamma'_n\supset \tilde\Gamma'_n\supset \Gamma'_{n+1}$.
\end{proof}

\section{Spectral Analysis of Heisenberg Odometer Actions}\label{sec:spectral}
 As was discussed in the introduction, Mackey \cite{M64} has shown that a $G$ odometer,  for $G$ any discrete, finitely generated and residually finite group, has discrete spectrum by showing that the action of $G$ on the inverse limit space $\iG$ is a sub-action of the compact group $\iG$ acting on itself by rotation.   In this section we analyze the exact nature of this decomposition.

It is clear that any one dimensional irreducible representation that appears comes from eigenfunctions of the group action.   In \cite{CP08} the authors show that the eigenvalues of a $G$ odometer are those characters $\phi:G\rightarrow S^1$ of the group $G$ for which $\phi(\gamma)=1$ for all $\gamma\in\Gamma_n$ for some $n$ and that the functions $f=\sum_{\gamma\in G/\Gamma_n}\phi(\gamma)\chi_{[n;\gamma]}$ are the corresponding eigenfunctions.    Below we present an alternate proof that odometer actions have discrete spectrum that allows us to identify the rest of the representations that occur in the decomposition explicitly and in terms of the irreducible, finite dimensional representations of the group $G$, as opposed to the group $\iG$.  

In the case of the discrete Heisenberg group $H$ this approach allows us to say much more.   An easy computation shows that if $\phi:H\rightarrow S^1$ is a character, then $\phi$ only depends on the first two components $(x,y)$.  Thus the one dimensional representations of a Heisenberg odometer are determined entirely by the eigenvalues of its associated $\mathbb Z^2$ odometer.   We describe below the eigenvalues of a $\mathbb Z^2$ odometer on $\iA$ in terms of the sequence of matrices $A_n$ and using the results of Section~\ref{sec:heisclassify} we identify explicitly those finite dimensional, irreducible representations of the Heisenberg group which arise in the spectral decomposition of the $H$ odometer action on $\iH$ as a function of the triples $(A_n,m_n,i_n)$ defining the subgroups $\Gamma_n$.

 \subsection{Discrete spectrum of general odometer actions}
  Let $G$ be a discrete, finitely generated, and residually finite group.  Fix a $G$ odometer on $\iG$.  Let $U:G\times\ L^2(\iG,\mu)\ra L^2(\iG,\mu)$ denote the induced unitary operator of the odometer action defined by 
 \begin{equation}
U(g)(f(\mathbf x))=f(g^{-1}\mathbf x).
\end{equation} 

 \begin{thm}\label{discrete}
$U$ admits a decomposition into finite dimensional irreducible representations of $G$ of the form
 $U=\bigoplus_{k=1}^\infty U_k$
 with the property that there exists an increasing sequence $k_n\rightarrow\infty$ so that for all $n$,  $\bigoplus_{k=1}^{k_n}U_k$ is equivalent to the unique irreducible decomposition of the regular representation of $G/\Gamma_n$.
\end{thm}
 
Before proving the result, we introduce some additional notation. Let $\Sigma$ be the Borel $\sigma$-algebra on $\iG$ generated by all cylinder sets $[n;\gamma]$, $n\ge 1$ and $\gamma\in G/\Gamma_n$; for each $n\ge 1$, let $\Sigma_n$ be the $\sigma$-algebra generated by the $n^{th}$-stage cylinder sets, and let $\mu_n=\mu_{|\Sigma_n}$ be the induced probability measure on the cylinder sets, which is normalized counting measure.    Let $X_n=(\iG, \Sigma_n, \mu_n)$.  If $g\in \Gamma_n$, then $U(g)(\chi_{[n;\gamma]})=\chi_{[n;g\gamma]}=\chi_{[n;\gamma]}$, so the unitary operator $U$ induces a representation $U^{(n)}$ of $G/\Gamma_n$ on  $L^2(X_n)$. By using the  natural unitary isomorphism between the finite dimensional spaces $L^2(G/\Gamma_n)$ and $L^2(X_n)$ given by $\chi_\gamma\mapsto \chi_{[n;\gamma]}$,  one easily checks that $U^{(n)}$ is equivalent to the regular representation of $G/\Gamma_n$ over $L^2(G/\Gamma_n)$. 

This equivalence allows us to use the machinery of regular representations of finite dimensional groups to prove Theorem~\ref{discrete} constructively.  In particular we will use two classical results which we state together in the following theorem, reformulated in our context.

\begin{thm}\label{Mas}{\em (Maschke's Theorem)}\cite{DF}  For every $n$, 
$L^2(X_n)=\bigoplus_{1\le k\le {k_n}} \mF^n_k$ where each $\mF^n_k$ is an irreducible, $U^{(n)}$ invariant subspace.  Furthermore, suppose that $V\subset L^2(X_n)$ is a $U^{(n)}$ invariant subspace.  Then $L^2(X_n)$ admits a decomposition of the form
$V\oplus \mF^n_{k_1}\oplus \cdots \oplus \mF^n_{k_j}$
for some sub collection of the subspaces $\mF^n_k$.
\end{thm}
It is obvious that $L^2(X_n)\subset L^2(X_{n+1})$. Moreover, a compatibility relation exists between the $G/\Gamma_{n}$-representation, $U^{(n)}$ and the $G/\Gamma_{n+1}$-representation, $U^{(n+1)}$.

\begin{lem}\label{lem-comp}
For every $\gamma\in G/\Gamma_{n+1}$,  $U^{(n+1)}(\gamma)$ restricted to $L^2(X_n)$ coincides with $U^{(n)}({\pi_n(\gamma)})$.
\end{lem}

\begin{proof}
It is sufficient to verify the condition for a given $\bar\gamma\in G/\Gamma_n$ and the associated characteristic function $\chi_{[n;\bar\gamma]}$:
\begin{eqnarray*}
U^{(n)}({\pi_n(\gamma)})\chi_{[n;\bar\gamma]}&=&\chi_{[n;\pi_n(\gamma)\bar\gamma]}=\sum_{\hat \gamma\in\pi_n^{-1}(\pi_n(\gamma)\bar\gamma)} \chi_{[n+1;\hat \gamma]}=\sum_{\hat \gamma\in\gamma\pi_n^{-1}(\bar\gamma)} \chi_{[n+1;\hat \gamma]}\\
&=&\sum_{\tilde \gamma\in\pi_n^{-1}(\bar\gamma)} \chi_{[n+1;\gamma\tilde \gamma]}
=U^{(n+1)}({\gamma})\chi_{[n;\bar\gamma]}.
\end{eqnarray*}
Here we used the fact that $\pi_n^{-1}(\pi_n(\gamma)\bar\gamma)=\gamma\pi_n^{-1}(\bar\gamma)$ which follows from the definition of $\pi_n$. Indeed,
\begin{equation*}
\begin{split}
\hat \gamma\in\pi_n^{-1}(\pi_n(\gamma)\bar\gamma)&\Leftrightarrow\pi_n(\hat \gamma)=\pi_n(\gamma)\bar\gamma\Leftrightarrow \bar \gamma= \pi_n(\gamma^{-1})\pi_n(\hat\gamma) \\ &\Leftrightarrow\gamma^{-1}\hat\gamma\in \pi_n^{-1}(\bar \gamma)\Leftrightarrow\hat \gamma\in \gamma\pi_n^{-1}(\gamma). 
\end{split}
\end{equation*}
\end{proof}

We are now ready to prove Theorem~\ref{discrete}.  Lemma~\ref{lem-comp} shows that $L^2(X_n)\subset L^2(X_{n+1})$ is a $U^{(n+1)}$ invariant subspace therefore
by Theorem~\ref{Mas} we have a decomposition
\[
L^2(X_{n+1})=L^2(X_n)\oplus \mF_{k_n+1}^{n+1}\oplus\cdots\oplus \mF_{k_{n+1}}^{n+1}
\]
into $U^{(n+1)}$ (and therefore $U$) invariant subspaces. 

Lemma~\ref{lem-comp} also implies that the decomposition of $L^2(X_n)$ into $U^{(n)}$ invariant irreducible subspaces is a decomposition into $U^{(n+1)}$ invariant and irreducible subspaces, therefore 
$\mF_1^n\oplus\cdots\oplus\mF_{k_n}^n\oplus\mF_{k_n+1}^{n+1}\oplus\cdots\oplus\mF_{k_{n+1}}^{n+1}$
is a decomposition of $L^2(X_{n+1})$ into  $U^{(n+1)}$ (and thus $U$) invariant, irreducible, finite dimensional subspaces.

So for all $n$, by relabelling, we have the  decomposition
$L^2(X_n)=\bigoplus_{1\le k\le{k_n}}\mF_k$
such that $U_k=U_{|_{\mF_k}}$ is an irreducible,  finite dimensional representation of $G$.  Recalling the fact that $U^{(n)}$ is equivalent to the regular representation of $G/\Gamma_n$, the collection $U_{k}$ for $k=1,\cdots,k_n$ must include all representations of that finite group.  

To complete the argument note that the space
$\mathfrak V=\bigoplus_{k=1}^\infty\mF_k$
contains all characteristic functions $\chi_{[n;\gamma]}$.  
Furthermore, the collection of characteristic functions generates an algebra of continuous functions that separates points in $\iG$, is closed under complex conjugation and contains the constants.  Therefore, by  the Stone-Weierstrass Theorem $\mathfrak V$ contains the continuous functions on $\iG$, and by the density of the continuous functions in the $L^2$ norm, must therefore be all of $L^2(\iG)$.

If in addition we suppose that $G$ is a discrete, countable, and amenable group, recent developments in the field yield the 
following immediate corollary.
\begin{cor}\label{entro}
If $G$ is a discrete, countable, and amenable group then every $G$ odometer action has zero entropy.
\end{cor}
\begin{proof}
Suppose not.  It would then follow from \cite{OW} that $(\iG,\mu,G)$ has a Bernoulli factor.  By \cite{Avni} this factor would have countable Lebesgue spectrum, contradicting Theorem~\ref{discrete}.
\end{proof}
In \cite{CP08} the authors prove that every sub-odometer is the measure theoretic factor of an odometer.  The following result then follows immediately from Theorem~\ref{discrete} and Corollary~\ref{entro}.
\begin{cor}
Every $G$ sub-odometer has discrete spectrum.  If, in addition $G$ is a discrete, countable, and amenable group, then every $G$ sub-odometer has zero entropy.
\end{cor}

\subsection{Spectrum of $\mathbb Z^2$ odometers}
Let us fix a $\mathbb Z^2$ odometer on $\iA$. Let
$\phi(x,y)=e^{2\pi i(\alpha x+\xi y)}$
denote a character of $\mathbb Z^2$.    By an abuse of notation we refer to the pair $(\alpha,\xi)$, rather than the associated character, as an eigenvalue of the action.  We further only concern ourselves with $(\alpha,\xi)\in\mathbb T^2$. For the remaining of this section we let the quotient notation, i.e. $\mathbb Z^2/A \mathbb Z^2$, mean a specific complete collection of coset representatives. 

Cortez \cite{Co06} shows that a pair $(\alpha,\xi)$ is an eigenvalue of the odometer action if and only if there exists $n$ so that
\begin{equation}\label{eval1}
(\alpha,\xi)^T\in\left((A_n^T)^{-1}\mathbb Z^2\right)/\mathbb Z^2.
\end{equation} 
An easy computation shows that (\ref{eval1}) holds if and only if there exists a vector $\mathbf{v}\in \mathbb Z^2/A^T_n\mathbb Z^2$
such that
\begin{equation}\label{eval}
(\alpha,\xi)^T=(A_n^T)^{-1}\mathbf{v}.
\end{equation}
It is therefore clear that there are $\det(A_n)$ distinct eigenvalues of the odometer action which are associated to the $n$th stage of the construction of the odometer.  Denote the set of vectors $(\alpha,\xi)$ satisfying \eqref{eval} by $\mathcal E(A_n)$ and fix $n$. For ease of notation we suppress the subscript $n$ and let the reader determine from the context whether the eigenvalues are being represented as row or column vectors.   

Let $A=\begin{pmatrix} a & b\\ c & d\end{pmatrix}$.  If $b=c=0$ then it follows from (\ref{eval}) that there will be $ad$ eigenvalues of the form $(\frac{j}a,\frac{j'}{d})$, $j\in\{0,\dots,a-1\}$, $j'\in\{0,\dots,d-1\}$.   To determine $\mathcal E(A)$ in the case of non-diagonal matrices we use (\ref{eval1}) to write
\begin{equation}\label{genform}
(\alpha,\xi)=\bigg(\frac{du-cv}{\det A},\frac{-bu+av}{\det A}\bigg)
\end{equation}
for some $(u,v)\in\mathbb Z^2$.
If the rows of $A$ have relatively prime entries, i.e. $gcd(a,b)=gcd(c,d)=1$, then the numerator of the first component can be made to take any value $0,\dots,\det A-1$, taken modulo $\det A$.   Therefore  there are exactly $\det A$ choices for $\alpha$, each with a corresponding $\xi$.  But there are $\det A$ pairs possible, so this exhausts all possible pairs that can occur, meaning all possible $\xi$ have now appeared.
So given an $\alpha$ there can only be a unique $\xi$ such that $(\alpha,\xi)$ is an eigenvalue.

Notice in the above two cases, for any $\alpha$ that appears in the spectrum of the odometer, the number of $\xi$ that can pair with it is the same for any $\alpha$.  It will follow from our work below that this holds in general.   In particular, given any matrix $A$ we can decompose it into the product of a diagonal matrix with a matrix whose rows have relatively prime entries:
\begin{equation*}
A=\begin{pmatrix} a & b \\ c & d\end{pmatrix}=\Delta\cdot\hat A=\begin{pmatrix} gcd(a,b) & 0\\ 0 & gcd(c,d)\end{pmatrix}\begin{pmatrix}\hat a & \hat b \\ \hat c & \hat d
\end{pmatrix}.
\end{equation*}
The set $\mathcal E(A)$ 
can be described in a nice geometric fashion using the matrix $\hat A$.  Using \eqref{eval} we have
$\mathcal E(A)=\Delta^{-1}(\hat A^T)^{-1}\left(\mathbb Z^2/(\Delta\hat A\mathbb)^T \mathbb Z^2\right)$.
On the other hand,
\begin{equation*}
\mathbb Z^2/(\Delta\hat A)^T\mathbb Z^2=\bigcup_{{0\le j<gcd(a,b)}\atop{0\le j'<gcd(c,d)}}\mathbb Z^2/\hat A^T\mathbb Z^2+\hat A^T\begin{pmatrix}j\\j'\end{pmatrix}
\end{equation*}
and therefore taking addition modulo $\mathbb Z^2$
\begin{align*}
\mathcal E(A)&=\bigcup_{{0\le j<gcd(a,b)}\atop{0\le j'<gcd(c,d)}}  \Delta^{-1}(\hat A^T)^{-1}\left(\mathbb Z^2/\hat A^T\mathbb Z^2\right)+\Delta^{-1}\begin{pmatrix}j\\j'\end{pmatrix}
=\Delta^{-1}\left(\mathcal E (\hat A)\right)+\mathcal E(\Delta).
\end{align*}

So given any $\mathbb Z^2$ odometer, there is now a clear algorithm for listing the eigenvalues corresponding to a fixed stage $n$.    Using \eqref{genform} first find $(u,v)\in\mathbb Z^2$ such that
$\hd u -\hc v=1$ and let $\hat \ell=-\hb u+\ha v$.
The following is a complete list (without repetition) of all $\det A$ eigenvalues of the $\mathbb Z^2$ action from this stage:
\begin{equation*}
\bigcup_{{0\le q<\det\hat A}\atop{{0\le j<gcd(a,b)}\atop{0\le j'<gcd(c,d)}}}\bigg(\frac{q+j\det\hat A}{gcd(a,b)\det\hat A},\frac{q\hat \ell+j'\det\hat A}{gcd(c,d)\det\hat A} \bigg).
\end{equation*}
Note that the since $A_n\mathbb Z^2<A_k\mathbb Z^2$ for all $k<n$, this collection will include all previous stage eigenvalues as well.

\subsection{Spectral analysis of Heisenberg odometers}
The main result of this section characterizes which finite dimensional irreducible representations of the Heisenberg group can arise in the spectrum of a given Heisenberg odometer.
Using standard inducing techniques from representation theory, Indukaev has shown in \cite{Induk} that for a given dimension $p$, 
all $p$ dimensional, irreducible representations of the discrete Heisenberg group $H$ can be parametrized by three scalars 
$(\alpha,\xi,\eta)\in \mathbb T^3$ and have the following structure:
\begin{equation}\label{rep}
U_{(x,y,z)}:\epsilon_j\mapsto e^{2\pi i(y\xi+(z+jy)\eta+[\frac{x+j}p]\alpha)}\epsilon_{(j-x)\!\!\!\!\mod p}
\end{equation}
where $(x,y,z)\in H$, $\eta\in\mathbb T^1$ is an irreducible fraction of the form $\frac{\ell}p$, $(\alpha,\xi)\in\mathbb T^2$ is arbitrary, $\epsilon_j$ denotes the projection onto the $j$th coordinate in a finite $p$-dimensional space, $j=0,\cdots,p-1$, and for $x\in\mathbb R$, $[x]$ denotes the integer part of $x$.  Conversely, given any such triple $(\alpha,\xi,\frac{\ell}p)$, the operator defined by \eqref{rep} gives a $p$ dimensional, irreducible representation of $H$.  
In what follows, for ease of notation, we say that $(\alpha,\xi,\frac{\ell}p)$ lies in the spectrum of an $H$ odometer action if the corresponding $p$ dimensional, irreducible representation of $H$ occurs in its spectral decomposition.

\begin{thm}
Let $\{\Gamma_n\}$, defined by $\{(A_n,m_n,i_n)\}$, be a sequence of normal subgroups of $H$ giving rise to a Heisenberg odometer action. 
Then $(\alpha,\xi,\frac{\ell}p)\in \mathbb T^3$, with $gcd(\ell,p)=1$, lies in the spectrum of the odometer action on $\iH$ if and only if there exists $n$ such that
$p$ divides $m_n$  and for all $(x,y)\in A_n\mathbb Z^2$
\begin{equation}\label{eq:want}
y\xi+\frac{i_n(x,y)\ell}p+\frac{x}{p}\alpha\in\mathbb Z.
\end{equation}
\end{thm}
\begin{proof}
Fix a Heisenberg odometer action on $\iH$ and let $(A_n,m_n,i_n)$ denote the triples associated with the subgroups $\Gamma_n$. 
Suppose that $(\alpha,\xi,\frac{\ell}p)$ is in the spectrum of $\iH$.  
Several easy observations follow from our proof of Theorem~\ref{discrete}.  First, since there is an inductive structure to the spectral decomposition of $L^2(\iH,\mu)$, there exists $n$ such that $(\alpha,\xi,\frac{\ell}{p})$ arises in the spectral decomposition of $L^2(H/\Gamma_n)$.  It also follows from the proof that the representations that occur in this decomposition are exactly the regular representations of the subgroup $H/\Gamma_n$.    Thus, for all $(x,y,z)\in\Gamma_n$, the representation given in \eqref{rep} with this triple reduces to the identity operator on $\Gamma_n$, meaning that we have
\begin{equation}\label{eq:inv}
y\xi+(z+jy)\frac{\ell}p+\bigg[\frac{x+j}p\bigg]\alpha\in\mathbb Z\qquad\text{for $j=0,\cdots,p-1$}.
\end{equation}
Since $(0,0,m_n)\in\Gamma_n$ it follows immediately that $p$ divides $m_n$.  Recall from Proposition~\ref{prop:normal} that it is necessary that $m_n$ divides all the entries of $A_n$ so $[\frac{x+j}{p}]=\frac{x}{p}$ and
 $\frac{jy}p\in\mathbb Z$.  Thus for $(x,y,z)\in\Gamma_n$  \eqref{eq:inv} reduces to  requiring
\begin{equation}\label{interm}
y\xi+\frac{z\ell}p+\frac{x}{p}\alpha=y\xi+\frac{(sm_n+i_n(x,y))\ell}{p}+\frac{x}{p}\alpha\in\mathbb Z
\end{equation}
for all $s\in\mathbb Z$.  The fact that $p$ is a divisor of $m_n$ now yields \eqref{eq:want}.

Suppose now that $(\alpha,\xi,\frac{\ell}p)$ is such that there exists $n$ for which $p$ divides $m_n$ and \eqref{eq:want} holds.  Using \eqref{interm} it follows immediately that \eqref{eq:inv} holds and thus the representation given in \eqref{rep} is an irreducible representation of $H/\Gamma_n$, therefore it must occur in the regular representation of $H/\Gamma_n$ (\cite{DF}).  It now follows from our proof of Theorem~\ref{discrete} that $(\alpha,\xi,\frac{\ell}p)$ occurs in the spectrum of $\iH$.
\end{proof}
\begin{rem}
If $p=1$, then the only non-zero irreducible fraction of the form $\frac{\ell}p$ comes from $\ell=0$ so in \eqref{eq:inv} we have $j=0$ and since $z\in\mathbb Z$ the one dimensional representations that occur have the form
 $e^{2\pi i(y\xi+ x\alpha)}$.
Thus, as we noted before, we see that the one dimensional representations that occur in the decomposition are precisely those that come from pairs $(\alpha,\xi)$ that are eigenvalues of the associated $\mathbb Z^2$ odometer.
\end{rem}

\subsection{Examples}
In the case of a flat odometer, without loss of generality we can assume $i_n(x,y)=0$ for all $n$, and we can conclude that the triples that occur are exactly those $(\alpha,\xi,\frac{\ell}{p})$ where $(\frac{\alpha}p,\xi)$ is an eigenvalue of the associated $\mathbb Z^2$ odometer.
Otherwise the values of $\xi$ and $\alpha$ depend on the structure of the functions $i_n$.  
Below we compute some examples to show 
the effect of these functions $i_n$ on the spectrum of the odometer. 

Consider the sequence of matrices $A_n=\begin{pmatrix} k_n & 0\\ 0 & k_n\end{pmatrix}$ and $m_n=k_n$ where the sequence $k_n$ as defined in Example~\ref{ex:kn}.  For any choice of the sequence $i_n$ for which the triple $(A_n,m_n,i_n)$ defines a Heisenberg odometer, the associated $\mathbb Z^2$ odometer has eigenvalues of the form $(\frac{j}{k_n},\frac{j'}{k_n})$ for $j,j'\in\{0,\cdots,k_n-1\}$.
Fix a stage $n$, assume $p$ divides $k_n$ and that $\ell$ is relatively prime to $p$. We consider three different choices for the sequence $i_n$.
\begin{enumerate}[(a)]
\item Let $i_n=0$.  The $p$-dimensional representations that occur will be  those corresponding to the triples
$\displaystyle{\bigg(\frac{pj}{k_n},\frac{j'}{k_n},\frac\ell{p}\bigg)}$.\\
\item Let $i_n=\displaystyle\frac{x}{k_n}$.  This gives rise to the odometer constructed in Example~\ref{ex:kn} where we showed that it is not flat so we expect a different set of representations to occur.  Indeed, the triples
$(\alpha,\xi,\frac{\ell}p)$ that can occur now have to satisfy:
\begin{align*}
y\xi+\frac{i(x,y)\ell}{p}+\frac{x}{p}\alpha&=y\xi+\frac{x\ell}{pk_n}+\frac{x}p\alpha=y\xi+x\bigg(\frac{\ell+k_n\alpha}{pk_n}\bigg)\in\mathbb Z.
\end{align*}
Therefore the representations that occur will be those corresponding to the triples
$\displaystyle{\bigg(\frac{jp-\ell}{k_n},\frac{j'}{k_n},\frac{\ell}{p}\bigg)}$.\\
\item Let $i_n=\displaystyle\frac{y}{k_n}$.  An argument parallel to that given in Example~\ref{ex:kn} shows that this choice of $i_n$ gives rise to an odometer action which is not flat.   The triples that occur must satisfy 
\begin{equation*}
y\xi+\frac{y\ell}{k_np}+\frac{x}{p}\alpha=y\bigg(\xi+\frac{\ell}{k_np}\bigg)+x\frac{\alpha}p\in\mathbb Z
\end{equation*}
which gives us triples of the form
$\displaystyle{\bigg(\frac{pj}{k_n},\frac{j'p-\ell}{k_np},\frac{\ell}{p}\bigg)}$.
Thus, this last example is not flat and also not conjugate to the odometer from Example~\ref{ex:kn}.
\end{enumerate}
\section{Acknowledgements}  The authors would like to thank Vitaly Bergelson, Jeff Bergen, Tomasz Downarowicz, Len Krop, Sasha Leibman, and E. Arthur Robinson for valuable conversations and in particular for bringing to our attention reference \cite{Induk}.  The first author's research was partially supported by a Western Connecticut State University sabbatical leave.  Both the second and third author's research was partially supported by a DePaul University College of Liberal Arts and Sciences Summer Research Grant.  The third author's research was also partially supported by NSF grant DMS-0703421.  Finally, the first author would like to thank DePaul University for its hospitality and support for the duration of this project.

\end{document}